\documentclass[12pt]{amsart}
\usepackage{amsmath}
\usepackage{amssymb}
\usepackage{amsfonts}
\usepackage{mathrsfs}
\usepackage{graphicx}
\usepackage{epsfig}
\usepackage[T1]{fontenc}
\numberwithin{equation}{section}       

\theoremstyle{plain}
\newtheorem{theorem}{Theorem}[section]

\newtheorem{prop}{Proposition}[section]
\newtheorem{conj}[prop]{Conjecture}
\newtheorem{coro}[prop]{Corollary}
\newtheorem{lemma}[prop]{Lemma}

\newtheorem{definition}[prop]{Definition}

\newtheorem{remark}[prop]{Remark}
\newtheorem{exam}[prop]{Example}

\newtheoremstyle{citing}
  {3pt}
  {3pt}
  {\itshape}
  {}
  {\bfseries}
  {.}
  {.5em}
  {\thmnote{#3}}

\theoremstyle{citing}

\DeclareMathAlphabet{\mathpzc}{OT1}{pzc}{m}{it} 

\newcommand{\C}{\mathbb{C}}
\newcommand{\D}{\mathbb{D}}

\newcommand{\cC}{\mathcal{C}}

\newcommand{\cF}{\mathcal{F}}

\newcommand{\cK}{\mathcal{K}}

\newcommand{\cN}{\mathcal{N}}

\newcommand{\cV}{\mathcal{V}}

\newcommand{\teta}{\widetilde{\teta}}

\newcommand{\eps}{\varepsilon}

\newcommand{\dist}{d}

\DeclareMathOperator{\diam}{diam}

\DeclareMathOperator{\supp}{supp} 

\begin{document}

\title[]{On invariant line fields of rational functions}

\author{Genadi Levin}

\address{Institute of Mathematics, The Hebrew University of Jerusalem, Givat Ram,
Jerusalem, 91904, Israel}

\email{levin@math.huji.ac.il}








\date{\today}




\maketitle

\begin{abstract} Adopting the approach
of \cite{L} we study rational function carrying invariant line fields on the Julia set. In particular,
we show that under certain weak conditions all possible measurable invariant line fields of a rational function on its Julia set are determined, in a precise sense, by a finite list of canonically defined
holomorphic line fields.
\end{abstract}

\section{Introduction}

There are two central conjectures in the dynamics of rational maps of the Riemann sphere.
\begin{conj}\label{conj:f} (Density of hyperbolicity, or Fatou). Any rational function is approximated by hyperbolic rational functions of the same degree.
\end{conj}

In turn, by \cite{MSS}, \cite{MS}, this conjecture would follow from a stronger:
\begin{conj}\label{conj:lf} (No invariant line field). A rational map $f$ carries no invariant line field on its Julia set $J$, except when $f$ is double covered by an integral torus endomorphism.
\end{conj}
For the latter, exceptional maps (called also flexible Lattes maps), see e.g. \cite{Mlattes}.
These are explicitly described critically finite rational functions with the Julia set the whole Riemann sphere.
In particular, this conjecture claims that there is no an invariant line field on $J$ if the Fatou set $F:=\hat\C\setminus J$ is non-empty,
e.g., if $f$ is a polynomial.

Recall that a {\it line field} on $J$ is a Beltrami differential $\mu(z)d\bar{z}/dz$ where $\mu\in L_\infty(\hat\C)$, such that $|\mu|=1$ on a subset of $J$ of positive measure and $0$ elsewhere. A line field is {\it holomorphic} if locally $\mu=|\phi|/\phi$ where $\phi$ is a holomorphic function.
An {\it invariant line field} on $J$ is a line field $\mu(z)d\bar{z}/dz$ where $\mu$ is a fixed point
of the pushforward operator
$$f^*: \nu\mapsto  (\nu\circ f) \frac{|f'|^2}{(f')^2}.$$
In this case, we say for short that $f$ {\it carries an invariant line field} $\mu$.

Conjecture \ref{conj:lf} allows one to move the focus from families of rational maps to the dynamics of an individual map.
For the quadratic family $f_c(z)=z^2+c$, Conjectures \ref{conj:f}-\ref{conj:lf} are equivalent and mean that the interior of the Mandelbrot set $M=\{c\in\C: J(f_c) \mbox{ is connected}\}$ consists of hyperbolic maps.
A stronger, MLC conjecture says that $M$ is locally connected \cite{dhorsay}.

A criterion for the absence of invariant line fields was developed in \cite{mcm}, \cite{mcm1}, see also \cite{shen}.
Following it, all "robust" infinitely-renormalizable polynomials have no invariant line fields, see \cite{mcm}, \cite{inou}, where "robustness" is defined through
the geometry of Riemann surfaces associated to the postcritical set.
A closely related to this is the property of an infinitely renormalizable map to have so-called a priori, or complex bounds which means a pre-compactness of an infinite sequence of renormalizations \cite{Su}; a priori bounds usually lead to important consequences.
A lot of efforts have been made to establish such kind of bounds for various classes of maps, see e.g. \cite{lICM} for an overview.
On the other hand,
there exist many, in particular quadratic, maps which are not robust and/or having no a priori bounds, see \cite{Mildh}, \cite{CP}, \cite{blot}
and references therein.





\



Let us outline a type of results of the paper and then state it more precisely for the quadratic family.

Throughout the paper, $area(A)$ is the Lebesgue measure of $A\subset\C$ and $|dx|^2$ is the area element at the point $x\in\C$
(sometimes we omit $|dx|^2$ in the integral though).

Let
$$P=\overline{\{f^n(c), n=1,2,..., \mbox{ over all critical points } c\in\hat\C \mbox{ of } f\}}$$
be the postcritical set of $f: \hat\C\to \hat\C$.
Recall that $P$ contains all attracting and parabolic cycles of $f$ as well as the boundaries of rotation domains.
Let
$P_J=P\cap J$.
\begin{definition}\label{def:basic}
A subset $K\subset J$ is a fundamental set if:
\begin{enumerate}
\item $P_J\subset K$ and $f(K)\subset K$,
\item $area(K)>0$, $area(K')>0$ where $K':=f^{-1}(K)\setminus K$, and
$$area(J\setminus \{K\cup\cup_{j=0}^\infty f^{-j}(K')\})=0.$$
\end{enumerate}
\end{definition}
If $K$ is fundamental, then $K, f^{-j}(K'), j=0,1,...,$ form a partition of $J$ up to a null (measure) set.


As it follows almost immediately from \cite{ly}, \cite{mcm}, assuming $P\neq J$, in all interesting for us cases any small neighborhood of the postcritical set of $f$ contains a fundamental set, see Lemmas \ref{kk'}, \ref{kk'-c}.
Another natural choice of fundamental sets are orbits of small Julia sets of infinitely-renormalizable unicritical polynomials, see Example \ref{ex:renorm}.


To a given rational map $f$, we assign a finite collection of canonical functions $\{r\}$ which are holomorphic outside of $P$ union the rotation domains of $f$. Each $r$ defines a holomorphic line field $\sigma=|r|/r$. It turns out that they "discover" all possible invariant line fields which are supported in $J\setminus \cup_{i\ge 0} f^{-i}(P)$.
Namely, for any invariant line field $\mu$ with its support $E_\mu\subset J\setminus \cup_{i\ge 0} f^{-i}(P)$ there is some $\sigma$ as above such that
$\mu$ and $\sigma$ are close in the following sense.
Given a fundamental set $K$, let
$\Delta_K=(\mu-\sigma)\circ R_{K'}$ where $R_{K'}$ is the first return map to $K'$.
Then
$$\Delta_K\to 0 \mbox {  in measure on } E_\mu \mbox{ as } K \mbox{ shrinks to } P.$$

To define functions $r$, we use the {\it pullback, or Ruelle-Thurston operator}:
$$(T_f g)(x)=\sum_{y: f(y)=x} \frac{g(y)}{(f'(y))^2}$$
whenever the right-hand side is defined.
The pushforward operator $f^*$ is dual to $T$ meaning that
$\int_J g f^*\nu =\int_J \nu T_f g$, for $g\in L_1(J)$, $\nu\in L_\infty(J)$ whenever both sides make sense.




\


\


For the quadratic family  $f_c(z)=z^2+c$,
we let
$$r(z)=(1-T_{f_c})^{-1}\frac{1}{z-c}.$$
We show below (in a more general setting, see Propositions \ref{psi} and \ref{psi-c}) that $r$ is a well-defined holomorphic function outside of the postcritical set
$P_c$
of $f_c$ union the rotation domains $\cN_c$ of $f_c$.

The following remark gives some impression how the function $r$ looks like (it will not be used in the rest of the paper though).
\begin{remark}\label{r:expl}
By Proposition \ref{psi} of Sect \ref{AS} and explicit formulas of \cite{L}, one can see that, for every $x\in\C\setminus (P_c\cup \cN_c)$:
$$r(x)=\lim_{\lambda\nearrow 1} D(\lambda)^{-1}\sum_{k=0}^\infty \frac{\lambda^k}{(f_c^k)'(c)}\frac{1}{x-f_c^k(c)},$$
where
$$D(\lambda)=\sum_{k=0}^\infty \frac{\lambda^k}{(f_c^k)'(c)}$$
and all series converge in $\D=\{|\lambda|<1\}$ by \cite{LPS}.
\end{remark}
Assume that $f_c$ is infinitely renormalizable (otherwise, by Yoccoz's results \cite{Yfields}, $f_c$ carries no invariant line fields on its Julia set).
Let $p_n\to\infty$ be periods of simple renormalizations of $f_c$ \cite{mcm}.
Given $n\ge 1$, let $J_n$ be the cycle of "small" Julia sets corresponding to the period $p_n$, that is,
$J_n=\cup_{i=0}^{p_n-1} f_c^i(J_n(0))$ where $J_n(0)\ni 0$ is the small Julia set containing zero.
Let finally $P_n=P\cap J_n(0)$. Then sets $f_c^i(P_n)$, $j=0,1,...,p_n-1$, are pairwise disjoint.



Let us fix any decreasing sequence of fundamental sets $(K_n)$ of $f_c$. Denote: $K'_n=f_c^{-1}(K_n)\setminus K_n$,
$K_\infty=\cap_{n>0}K_n$, $K'_\infty=f_c^{-1}(K_\infty)\setminus K_\infty$ and
$K_{\pm\infty}=\cup_{i\ge 0}f_c^{-i}(K_\infty)$.

Let
$$\sigma=\frac{|r|}{r}.$$
For every $n$, we define a {\it line field $\sigma_n d\bar z/dz$ supported on}
$$\cup_{i\ge 0}f_c^{-i}(K'_n\setminus K'_\infty)$$
as follows.
Let us restrict $\sigma$ to $K'_n\setminus K'_\infty$ and spread the restriction over all preimages thus letting
$\sigma_n=(f^i)^*(\sigma|_{K'_n\setminus K'_\infty})$ on $f^{-i}(K'_n\setminus K'_\infty)$, for all $i\ge 0$, and $\sigma_n=0$ elsewhere.




It is well-known that a unicritical polynomial $f_c$ can have at most one invariant line field $\mu$ on $J_c$, up to rotation, moreover, if $\mu$ is such a line field then
$\int_{J_c}\frac{\mu(w)}{w-c}|dw|^2\neq 0$ (\cite{MS}, \cite{L}, also Proposition \ref{fix}, Corollary \ref{c-neq} below). Replacing $\mu$ by $t\mu$ with some $|t|=1$
we can (and will) assume w.l.o.g., that
$$\hat\mu(c):=\int_{J_c}\frac{\mu(w)}{w-c}|dw|^2>0.$$

\begin{coro}\label{coro:quadr}

The following conditions are equivalent:
\begin{enumerate}
\item $f_c$ carries an invariant linefield $\mu$ such that
$$E_\mu:=supp(\mu)\subset J\setminus \cup_{i\ge 0} f^{-i}(K_\infty)$$
(e.g. this holds whenever the measure of $K_\infty$ is zero
\footnote{A rather weak sufficient condition for the postcritical set $P$ to be a null set is given in \cite{mcm}, Theorem 2.4 and Remarks followed; in particularly, this holds for robust polynomials});
\item
\begin{equation}\label{intk'nquadr}
\liminf_{n>0}\int_{K_n'\setminus K'_\infty}|r(x)| |dx|^2>0.
\end{equation}

Moreover, there exist and are equal
$$\lim_{n\to\infty}\int_{E_\mu\cap K_n'}|r(x)| |dx|^2=\lim_{n\to\infty}\int_{K_n'} r(x)\mu(x) |dx|^2
=\hat\mu(c)
>0.$$

\item the sequence $\{\sigma_{n}\}\subset L_\infty(J)$ converges in the weak* topology to the invariant line field $\mu$,


\item

(a) a stronger convergence holds:
\begin{equation}\label{coro:inmeasure}
\sigma_n\to \mu \mbox{ as } n\to\infty \mbox{ {\bf in measure} on the support } E_\mu \mbox{ of } f.
\end{equation}


Equivalently, let
$$\Delta_n=(\mu-\sigma)\circ f^j \mbox{ on } f^{-j}(K_n'\cap E_\mu) \mbox{ for } j=0,1,2,...$$
Then
\begin{equation}\label{coro:delta}
\Delta_n\to 0 \mbox {  in measure on } E_\mu \mbox{ as } n\to\infty.
\end{equation}

(b) Assuming, additionally, that $K_\infty:=\cap_{n>0}K_n=P$, yet a stronger version of (\ref{coro:delta}) holds:

for every $m$ and every neighbourhood $V$ of $P_m$ such that $\overline{V}\cap (P\setminus P_m)=\emptyset$,
\begin{equation}\label{coro:delta0}
\Delta^{V}_n\to 0 \mbox {  in measure on } E_\mu \mbox{ as } n\to\infty
\end{equation}
where
$$\Delta^{V}_n=(\mu-{\bf 1}_{V}\times  \sigma)\circ f^j \mbox{ on } f^{-j}(K_n'\cap E_\mu) \mbox{ for } j=0,1,2,....$$
and ${\bf 1}_{V}$ is the indicator function of $V$.

\end{enumerate}

\end{coro}

For the proof of Corollary, see Sect \ref{PR}.

\


For the rest of the paper, let us fix $f$.
Naturally,
\begin{itemize}
\item
either the Fatou set of $f$ is non-empty,
\item
or the Julia set $J$ of $f$ is the whole Riemann sphere.
\end{itemize}
We deal with the former case in Sections \ref{MR}-\ref{PR}. The latter case is very similar, with differences described in Sect \ref{rat-c}.

Main results of the paper are Proposition \ref{propmain}, Theorem \ref{thm:main} (see also Corolarry \ref{coro:gen}) and
Theorem \ref{main-c}.

In the last Sect \ref{app}, as an application, we prove McMullen's theorem  (\cite{mcm}, Theorem 10.2) on the absence of invariant line fields for quadratic polynomials with unbranched complex bounds. Although the proof follows partly lines of \cite{mcm}, the  main difference is that we
don't use Lebesgue's almost continuity theorem which is a key point in McMullen's approach.

{\bf Acknowledgment.} The author would like to thank Benjamin Weiss for a helpful conversation.
Results of this work were reported at the conference "On geometric complexity of Julia sets V", Będlewo, Poland, July 21-26, 2024.

\section{Rational maps with non-empty Fatou set. Main Results}\label{MR}

In this section we assume that the Fatou set $F$ of $f$ is non-empty.
By Sullivan's No Wandering Domain theorem, any component of $F$ is (pre)periodic. Any periodic component
is either attracting or parabolic or rotation (Siegel or Hermann).

An important particular case is when $f$ is a polynomial.

If $f$ is not a polynomial,
after a conjugation by a Mobius transformation, we may assume, without loss of generality, that:

{\it point at $\infty$ lies in a periodic component $U_\infty$ of $F$, moreover, 
$a:=f(\infty)\neq \infty$ and $a$ is not a critical value of $f$.}

Let
$$P_a=\overline{\{f^n(a), n=1,2,...\}}.$$




%

Let $\cV=\{v_1,...,v_p\}$ be all pairwise different critical values of $f$
such that $v_i=f(c)\in\C$ for some $c\in\C$, $f'(c)=0$.

Let $T=T_f$ be the Ruelle-Thurston operatot of the map $f:\hat\C\to\hat\C$.
The following (resolvent) function plays a central role in the paper:
$$r_\psi=(1-T)^{-1}\psi.$$
  
{\it In the sequel, unless otherwise is explicitly stated, the function $\psi$ is always taken from the following list of test functions:
$$\cC=\{\frac{1}{x-v_1}, ..., \frac{1}{x-v_p}\},$$
if $f$ is a polynomial, and
$$\cC=\{\frac{1}{x-v_1}, ..., \frac{1}{x-v_p}, \frac{1}{x-a}, \frac{1}{(x-a)^2}, \frac{1}{(x-a)^3}\},$$
if $f$ is a rational function normalized as above.}

We show in Proposition \ref{psi} below that $r_\psi$ is a well-defined non-zero holomorphic function in an open set $\Omega$, for each $\psi\in\cC$, where:
$$\Omega=\C\setminus \{P\cup \cN\},$$
if $f$ is a polynomial, and
$$\Omega:=\C\setminus \{P\cup P_a\cup \cN\}$$
if $f$ is a rational functions (normalized as above)
where $\cN$ is the union of cycles of rotation domains of $f$.

Given $\psi\in \cC$,
let
$$\sigma^\psi(x)=\frac{|r_\psi(x)|}{r_\psi(x)}.$$
It is well-defined for each $x\in \Omega$ whenever $r_\psi(x)\neq 0$, i.e., outside of a discreet subset of $\Omega$ (at worst).





Before continuing let's show that $f$ has many fundamental sets (see Definition \ref{def:basic}) assuming that $P_j\neq J$:
\begin{lemma}\label{kk'}
Suppose that $area(J)>0$. Given a neighborhood $W$ of the set $P_J$ such that $area(\partial W)=0$,
let $K(W):=\{x\in W\cap J: f^n(x)\in W, n=1,2,...\}$. Assume that $area(J\setminus W)>0$.
Then $K(W)$ is a fundamental set.
\end{lemma}
\begin{proof}

As $J\neq \bar\C$, by a classical argument \cite{ly}, \cite{mcm}, for a.e. $x\in J$, $\omega(x)\subset P_J$.
Therefore, $area(K(W))>0$ and $\cup_{j=0}^\infty f^{-j}(K(W))$ is a subset of $J$ of full measure.
Now, $area(K'(W))=0$ would imply that $f^{-1}(K(W))$ is equal to $K(W)$, up to a null set, hence,
the same would hold for $\cup_{j\ge 0} f^{-j}(K(W))$ and $K(W)$ where the former set has also full measure in $J$.
Thus $\overline{K(W)}$ has full measure in $J$ yielding $J\subset \overline{W}$, a contradiction with conditions of the statement.
\end{proof}

Here is an example of fundamental sets which is relevant to Corollary \ref{coro:quadr} of the Introduction.
\begin{exam}\label{ex:renorm}
For an infinitely renormalizable $f_c(z)=z^d+c$ and each level $n>0$ of a simple renormalization, let $J_n$ be the orbit of small Julia set $J_n(0)$ of this level $n>0$ that contains $0$. Note that $J'_n=f_c^{-1}(J_n)\setminus K_n$ coincides with the union of all rotations of $J_n$ by angles
$2\pi k/d$, $k=1,...,d-1$ with the set $J_n(0)$ being removed. Hence, assuming $area(J_c)>0$, then $area(J'_n)>0$. Moreover, $J_n$ is a fundamental set because the orbit of a.e. point of $J_c$ is absorbed by $J_n$, by Theorem 8.2 of \cite{mcm}.
\end{exam}


In order to state main results, fix a decreasing sequence $K_n$, $n=1,2,...$ of fundamental sets.
Let $K'_n=f^{-1}(K_n)\setminus K_n$, $K_\infty=\cap_{n\ge 0} K_n$, $K'_{\infty}=f^{-1}(K_\infty)\setminus K_\infty$
and $K_{\pm\infty}=\cup_{i=0}^\infty f^{-i}(K_\infty)$.

For a given $\psi\in\cC$ and every $n$, let $\sigma^\psi_n$ be the following line field:
\begin{equation}\label{sigman}
\sigma^\psi_n(x)=\left \{
                      \begin{aligned}
                       &\sigma^\psi(x), && \text{ if }  x\in K_n'\setminus K'_\infty \\
                       &(f^i)^*\sigma^\psi(x), && \text{ if }  f^i(x)\in K_n'\setminus K'_\infty \text{ for some } i>0 \\
                       &0, && \text{ otherwise }
  \end{aligned} \right.
\end{equation}

\begin{remark}\label{remarksigma}

1. Functions $\sigma^\psi_n$ belong to the closed unit ball of $L_\infty(J)$ which is a sequentially weak* compact set.

2. Assume that $J$ carries an invariant line field $\mu$ with $\supp(\mu)\cap K_{\pm\infty}=\emptyset$.
If, for some $y\in \supp(\mu)$ and some $i\ge 0$, $f^i(y)\in K'_n$, then
$$\frac{\sigma^\psi_n(y)}{\mu(y)}=\frac{(f^i)^*\sigma^\psi(y)}{(f^i)^*\mu(y)}=\frac{\sigma^\psi(f^i(y))}{\mu(f^i(y))}.$$
\end{remark}

\begin{prop}\label{propmain}
The following conditions are equivalent.
\begin{enumerate}
\item $f$ carries an invariant linefield which is supported on (a subset of) $J\setminus K_{\pm\infty}$,
\item there exists $\psi\in \cC$ such that
\begin{equation}\label{intk'n}
\liminf_{n>0}\int_{K_n'\setminus K'_\infty}|r_\psi(x)| |dx|^2>0.
\end{equation}
\end{enumerate}
Furthermore, for every $\psi\in\cC$, such that (\ref{intk'n}) holds,
any weak* limit point $\nu$ of the sequence $\{\sigma^\psi_n(x)\}_{n=1}^\infty$ is a nonzero fixed point of the operator $f^*$
with the norm $||\nu||_\infty\le 1$ whose support is contained in $J\setminus K_{\pm\infty}$.

\end{prop}

Let us call an invariant line field $\mu$ on $J$ {\it ergodic} if its support $E_\mu$ is a completely invariant ergodic subset of $f:J\to J$
and $\mu$ is the unique invariant line field with support contained in $E_\mu$, up to a rotation, see Proposition \ref{fix}.

\begin{theorem}\label{thm:main}
(asymptotic extremality of the linefields $\sigma^\psi$)

Suppose that $\mu$ is an ergodic invariant line field with its support $E_\mu\subset J\setminus K_{\pm\infty}$.
Let us fix some $\psi\in\cC$ for which $\int_J \mu(x)\psi(x)|dx|^2\neq 0$ (by Proposition \ref{fix}, at leat one such $\psi$ always exists).
Replacing $\mu$ by $t\mu$ for some $|t|=1$,
one can assume that
$$\int_J \mu \psi >0.$$
For every $n$, let
$$\sigma^{\psi}_{E,n}=\sigma^\psi_n|_{E_\mu} \mbox{ on } E_\mu \mbox{ and } \sigma^{\psi}_{E,n}=0 \mbox{ on } \C\setminus E_\mu,$$
the restriction of the line field $\sigma^\psi_n$ to $E$.
Then:
\begin{enumerate}

\item the sequence $\{\sigma_{E,n}^{\psi}\}\in L_\infty(J)$ converges in the weak* topology to an invariant line field $\nu_0$ where $\nu_0=\mu$ on $E_\mu:=\supp(\mu)$ and $\nu_0=0$ on $\C\setminus E_\mu$,

\item the following limits exist and equal:
$$\lim_{n\to\infty}\int_{E_\mu\cap K_n'}|r_\psi(x)| |dx|^2=\lim_{n\to\infty}|\int_{K_n'} r_\psi(x)\mu(x) |dx|^2|
=|\int_J \mu(x) \psi(x)|dx|^2|>0.$$

\item
\begin{equation}\label{thm:inmeasure}
\sigma^{\psi}_{E,n}\to \mu \mbox{ as } n\to\infty \mbox{ in measure on } E_\mu.
\end{equation}

Equivalently, let
$$\Delta_n=(\mu-\sigma)\circ f^j \mbox{ on } f^{-j}(K_n'\cap E_\mu) \mbox{ for } j=0,1,2,...$$
Then
\begin{equation}\label{thm:delta}
\Delta_n\to 0 \mbox {  in measure on } E_\mu \mbox{ as } n\to\infty.
\end{equation}

\end{enumerate}
\end{theorem}
An almost immediate consequence of these result and their proofs is the following
\begin{coro}\label{coro:gen}
Given an invariant line field $\mu$ with its support $E_\mu\subset J\setminus \cup_{i\ge 0}f^{-i}(K_\infty)$,
let now $\psi\in L_1(J)$ be {\it ANY} function such that
\begin{equation}\label{mupsi}
\int\mu \psi>0.
\end{equation}
Then:
\begin{enumerate}\label{gen}
\item $$r_\psi=(1-T)^{-1}\psi$$
is defined a.e. on $J\setminus \{P\cup \cN\}$,
\item
$$\liminf_{n>0}\int_{K_n'\setminus K'_\infty}|r_\psi(x)| |dx|^2\ge \int \mu\psi |dx|^2>0.$$
In particular, for every $n$ large enough, the function
$$\sigma^\psi:=\frac{|\psi|}{\psi}$$
is well-defined on a subset of $K'_n$ of a positive measure,
\item the conclusions (1)-(3) of Theorem \ref{thm:main} hold for this $\psi$.
\end{enumerate}
\end{coro}
\begin{proof} Item (1) follows by repeating the proof of part (1) of Propositions \ref{psi}, \ref{psi-c}.
In turn, proofs of items (2), (3) repeat almost literally the ones of the first implication of Proposition \ref{propmain}
and of Theorem \ref{thm:main} accordingly.
\end{proof}
\begin{remark}\label{psivspsi}
Our choice of $\psi$ to be taken from the finite list $\cC$ has important advantages: first, one of $\psi\in\cC$ always satisfies the condition
(\ref{mupsi}), secondly, the corresponding resolvent function $r_\psi$ is holomorphic (which is crucial, for example, in Sect \ref{app}).
\end{remark}

\section{Parametrizing invariant line fields}\label{AS}
Here we prove a few preparatory results. Some of them can be found elsewhere (with different formulation or proof, or in less generality,
see comments below).

Recall that $\cV=\{v_1,...,v_p\}$ are all pairwise different critical values of $f$ which lie in the complex plane and are images
of critical points that are also in the complex plane.  
\begin{lemma}\label{T}
For any $x, z\in\C$ such that $x\neq a$, $x\neq f(z)$, $f'(z)\neq 0$ and $x\notin \cV$:
\begin{equation}\label{eqT}
T\frac{1}{z-x}=\frac{1}{f'(z)}\frac{1}{f(z)-x}+\sum_{j=1}^p\frac{L_j(z)}{x-v_j}-I(z,x),
\end{equation}
for some functions $L_j(z)$ and where:

$I(z,x)=0$ if $f$ is a polynomial; if $f$ is not a polynomial, then
\begin{equation}\label{I}
I(z,x)=\frac{p_0}{(a-x)^3}+\frac{p_1(z)}{(a-x)^2}+
\frac{p_{2}(z)}{(a-x)},
\end{equation}
where $p_i(z)$ are polynomials in $z$ with constant coefficients of degree at most $i=0,1,2$
\end{lemma}
\begin{remark}\label{l-rem}
In fact, $L_j(z)$ is a function which is meromorphic on the Riemann sphere and having poles precisely at
points $c\in\C$ such that $f'(c)=0$ and $v_j=f(c)\in\C$, see Lemma 2.1 of \cite{Lpert} where the normalization of $f$ at $\infty$ is different
but this part in the proof is local.
\end{remark}
\begin{proof}
For a fixed $z, x\in\C$ and $R>0$ big enough, let
$$I(z,x)=\frac{1}{2\pi i}\int_{|w|=R}\frac{1}{f'(w)(f(w)-x)(w-z)}dw.$$
$I(z,x)$ is independent of $R$ provided $R$ is big enough. On the one hand, for $x,z$ as in the conditions,
$$I(x,z)=-T\frac{1}{z-x}+\frac{1}{f'(z)}\frac{1}{f(z)-x}+\sum_{c\in\C: f'(c)=0, f(c)\in\C} I_c(z,x)$$
where the operator $T$ acts on the variable $x$ while
$$I_c(z,x)=\frac{1}{2\pi i}\int_{|w-c|=\epsilon}\frac{1}{f'(w)(f(w)-x)(w-z)}dw,$$
for some small enough $\epsilon>0$.
Repeating a calculation in the proof of Lemma 2.1 of \cite{Lpert}, one gets:
$I_c(z,x)=L_c(z)/(x-f(c))$. We put
$$L_j(z)=\sum_{c\in\C: f'(c)=0, v_j=f(c)}L_c(z).$$
On the other hand, calculating the integral $I(z.x)$ at $\infty$, we obtain that $I(z,x)=0$ if $f$ is a polynomial
while, if $f$ is a rational function with the chosen normalization such that $f(\infty)=a\neq \infty$,
we get (\ref{I}) (after a somewhat tedious calculation). Note that $f(w)=a+\frac{A_1}{w}+O(\frac{1}{w^2})$, as $w\to\infty$, where $A_1\neq 0$ 
because $\infty$ is not a critical point of $f$.
\end{proof}

Let

$$|T|(g)(x):=\sum_{y: f(y)=x} \frac{|g(y)|}{|f'(y)|^2}$$
be the associated to $T$ absolute value operator.

Given a function $\psi\in\cC$, recall that
$$r_\psi=(1-T)^{-1}\psi.$$
Let also
$$|r|_{\psi}=\sum_{i=0}^\infty |T|^i (\psi).$$




\begin{prop}\label{psi}


(1)
For each $\psi\in\cC$, the series
$|r|_{\psi}=\sum_{i=0}^\infty (|T|^i \psi)(x)$
converges locally uniformly to a continuous function in
the open set $\Omega$.

(2) for every $\psi\in\cC$, the function $r_\psi$ is well-defined and holomorphic in $\Omega$, and
$$r_\psi(x)=\sum_{i=0}^\infty (T^i \psi)(x)$$
where the series converges locally uniformly in $\Omega$.

\end{prop}
\begin{proof}


First, we prove that: 
\begin{equation}\label{locunif}
\sum_{i=0}^\infty |T|^i (\psi) \mbox{ converges locally uniformly in } \Omega.
\end{equation}
{\bf Case 1}. Let $x_0\in\Omega\cap F$. It follows from our normalization of $f$ and since $x_0$ lies either in an attracting basin without the corresponding attracting periodic point,
or in a parabolic basin, or in a preimage of such basins that, for some $\rho>0$,
all components of the set $f^{-i}(B(x_0,2\rho))$, $i=0,1,...$, are pairwise disjoint and uniformly away from $\infty$.
Therefore, by change of variables,
for $\psi(x)=(x-v_j)^{-1}$,
\begin{equation}\label{rpsiconv}
\int_{B(x_0, 2\rho)}\sum_{i=0}^\infty (|T|^i \psi)(x)|dx|^2=\int_{\cup_i f^{-i}(B(x_0, 2\rho))}|\psi(x)||dx|^2:=S<\infty.
\end{equation}
On the other hand, 
$$|T|^i \psi)(x)=\sum_{y: f^i(y)=x}\frac{|\psi(y)|}{|(f^i)'(y)|^2}.$$
Sine $B(x_0, 2\rho)$ is disjoint with the postcritical set $P$ so that every branch $g:=f^{-i}:B(x_0,2\rho)\to\C$
is well-defined and univalent,
by Koebe distortion theorem, there are some absolute constants $C_1, C_2>0$ such that
for all $x_1,x_2\in B(x_0,\rho)$, we have:
\begin{equation}\label{koebe}
\frac{|g'(x_1)|}{|g'(x_2)|}\le C_1 \mbox{ and } \frac{\dist(g(x_1), z)}{\dist(g(x_2), z))}\le C_2
\mbox{ for all } x_1,x_2\in B(x_0,\rho), z\notin g(B(x_0, 2\rho)).
\end{equation}
($\dist$ means the Euclidean distance).
By (\ref{koebe}), there exists an absolute $L>0$ such that $|\psi(g(x))|/|\psi(g(x_0))|\le L$ for all $x\in B(x_0, \rho)$, any
branch $g:=f^{-i}:B(x_0,2\rho)\to\C$ and any $\psi\in\cC$ (in fact, $L=C_2$ if $\psi(x)=1/(x-z)$, for $z\in V$ or $z=a$, and $L=C_2^3$ for
$\psi(x)=1/(x-a)^i$, $i=2,3$).
Again using (\ref{koebe}), we get, for $N_1<N_2$ and all $x, y\in B(x_0,\rho)$:
\begin{equation}\label{koebeunif}
\frac{\sum_{i=N_1}^{N_2} (|T|^i \psi)(y)}{\sum_{i=N_1}^{N_2} (|T|^i \psi)(x)}\le L C_1^2.
\end{equation}
Using that, we conclude from (\ref{rpsiconv}) that, for every $y\in B(x_0,\rho)$,
$$|r|_{\psi}(y)\le \frac{L C_1^2}{\pi\rho^2} S,$$
which, along with (\ref{koebeunif}) implies (\ref{locunif}) in $\Omega\setminus J$. 


{\bf Case 2}. Now, let $x_0\in J\cap \Omega$ and $\rho>0$ be so that $B(x_0, 2\rho)\subset \Omega$.
Let us show that there is $x_1\in B(x_0,\rho)\cap\Omega\setminus J$. Indeed, otherwise
$B(x_0,\rho)\setminus J$ would consist only points of rotation domains of $f$. As $f^M(B(x_0,\rho))\supset J$ for some $M$,
then the whole Fatou set $F$ would consist only rotation domains, a contradiction as there exist preimages
of rotation domains other than themselves. 

Thus one can pick $x_1\in B(x_0, \rho)\cap\Omega\setminus J$ for which we know already that
$|r|_\psi(x_1)$ exists. Then we proceed as in the proof of Case 1 to get (\ref{locunif}) near $x_0$. Together with Case 1, this proves (\ref{locunif}).

Since partial sums of
$|r|_{\psi}$ are continuous functions in $\Omega$ and
partial sums of $\sum_i T^i\psi$ are holomorphic functions in $\Omega$, (1)-(2) follow.

\end{proof}

Let $Fix(f^*)$ be a subspace in $L_\infty(J)$ of all fixed points of the operator $f^*$.

Denote, for every $\nu\in L_\infty(J)$,
$$\hat\nu(z)=\int_{J}\frac{\nu(x)}{z-x}|dx|^2, \hat\nu'(z)=\int_{J}\frac{\nu(x)}{(z-x)^2}|dx|^2, \hat\nu''(z)=\int_{J}\frac{\nu(x)}{(z-x)^3}|dx|^2,$$
where the first integral is the Cauchy transform of the measure $\nu(x)|dx|^2$ supported on a compact subset of $\C$ (so it always exists),
and the latter two integrals exist whenever $z\notin J$.

\begin{prop}\label{fix}


(1)
The following map $\cF: Fix(f^*)\to\C^{p+3}$:
$$\nu\in Fix(f^*)\mapsto F(\nu)=(\hat\nu(v_1),...,\hat\nu(v_p),\hat\nu(a),\hat\nu'(a),\hat\nu''(a))$$
is linear and injective.

If $f$ is a polynomial, we let here $\hat\nu(a)=\hat\nu'(a)=\hat\nu''(a)=0$.


(2) $Fix(f^*)$ is either trivial or a finite direct sum of $1$-dimensional spaces
spanned by ergodic invariant line fields:
there are completely invariant ergodic subsets $E_1,...,E_m$ of $J(f)$ with pairwise disjoint supports, such that each $E_i$ carries a unique (up to rotation) invariant linefield $\mu_i$ and every $\nu\in Fix(f^*)$ admits a unique representation
$\nu=\sum_{i=1}^m t_i\mu_i$ for some $(t_1,...,t_m)\in\C^m$.
\end{prop}
\begin{remark}\label{ref}
See also Proposition \ref{fix-c}. For part (1) of Propositions \ref{fix}=\ref{fix-c}, cf. \cite{L}, \cite{mak} (where a similar statement is proved by a different method assuming the critical points of $f$
are simple), and for part (2) of Propositions \ref{fix}-\ref{fix-c}, cf. \cite{MS}, Theorem 6.5.
\end{remark}
\begin{proof}
(1) It is obvious that $\cF$ is linear. To prove that it is injective, we use the following functional equation for the
Cauchy transform $\hat\mu(z)=\int\mu(x)|dx|^2/(z-x)$ of a given $\mu\in Fix(f^*)$, cf. \cite{L}. As $\mu\in L_\infty(\C)$ has a compact support
under the normalization of $f$,
it is well-known that $\hat\mu$ is a continuous function in $\C$, holomorphic off $J$, and $\hat\mu(\infty)=0$. 
By Lemma \ref{T} and since $\int\mu(x)T\frac{1}{z-x}|dx|^2=\int \frac{f^*\mu(x)}{z-x}|dx|^2=\int\frac{\mu(x)}{z-x}|dx|^2$,
one easily gets:
$$\hat\mu(z)=\frac{1}{f'(z)}\hat\mu(f(z))+\sum_{j=1}^p L_j(z)\hat\mu(v_j)-\hat I,$$
where $\hat I=0$ if $f$ a polynomial, and $\hat I=p_2(z)\hat\mu(a)+p_1(z)\hat\mu'(a)+p_0\hat\mu''(a)$ otherwise.
As $\mu\equiv 0$ if $\hat\mu\equiv 0$, it is enough to check that $\hat\mu\equiv 0$ whenever $\bar v=0$ where $\bar v=(\hat\mu(v_1),\cdots, \hat\mu(v_p), \hat\mu''(a), \hat\mu'(a), \hat\mu(a))$
(the latter $3$ coordinates in $\bar v$ are zero in the polynomial case).
So provided $\bar v=0$, we get from the functional equation that
$\hat\mu(z)=\frac{1}{f'(z)}\hat\mu(f(z))$ for all $z\in\C$. Since $\hat\mu$ is continuous on $\C$, holomorphic
in $F$ and $\hat\mu(\infty)=0$,
it's enough to show that $\hat\mu=0$ arbitrary close to each point of $J$. To this end, assume first that $F$ has a component $D$ with infinitely many different preimages $f^{-i}(D)$ and fix $z\in D$. Then every $a\in J$ is a limit point of preimages of $z$: there is $z_{-n_i}\to a$ such that $f^{n_i}(z_{-n_i})=z$.
On the other hand, $(f^{n_i})'(z_{-n_i})\to\infty$ as $n_i\to\infty$. Hence,
$\hat\mu(z_{-n_i})=\hat\mu(z)/(f^{n_i})'(z_{-n_i})\to 0$, and we are done.  If there is no component $D$ as above, then $J$ has a completely invariant component (for some iterate of $f$) which is not a rotation domain, and the proof in this case is very similar.
(Another argument that $\hat\mu\equiv 0$ is by noting that $\hat\mu(z)=\frac{1}{f'(z)}\hat\mu(f(z))$ implies $\hat\mu(z)=0$ along every repelling cycle of $f$ which, in turn, are dense in $J$.)

(2) First, let $E$ be a completely invariant ergodic subset of $J$ of positive measure and $\supp(\nu)=E$, for some $\nu\in Fix(f^*)\setminus \{0\}$.
As $|\nu|=|\nu\circ f|$ a.e. on $J$ and $\nu$ is measurable, $|\nu|$ is a constant function a.e. on $E$. Moreover,
as $Fix(f^*)$ is a linear space, it follows from here that the restriction of $Fix(f^*)$ on $L_\infty(E)$ is one-dimensional,
hence, is spanned by an invariant ergodic line field with its support to be $E$.
Assuming that $Fix(f^*)$ is not the null space, let $\nu_1,...,\nu_m$ be a basis of $Fix(f^*)$ where $m<\infty$ by part (1).
Now, if, for some $i$, $\supp(\nu_i)$ is not ergodic, then it splits into at most $m$ ergodic subsets
such that the restriction of $\nu_i$ on each of them, by the above, is proportional to an ergodic line field.
The collection of all such different ergodic line fields for all $i$ form the required basis of $Fix(f^*)$.


\end{proof}
\begin{coro}\label{c-neq}
For every $\nu\in Fix(f^*)$, $\nu\neq 0$, the vector
$$(\hat\nu(v_1), ..., \hat\nu(v_p), \hat\nu(a), \hat\nu'(a), \hat\nu''(a))\neq 0.$$
\end{coro}
Recall that $\hat\nu(a)=\hat\nu'(a)=\hat\nu''(a)=0$ in the polynomial case.

Notice
\begin{lemma}\label{psinonint}
If $J$ carries an invariant line field $\mu$ then $r_\psi\notin L_1(J)$ whenever $\int_J \mu \psi \neq 0$, hence, for at leat some $\psi\in\cC$.
\end{lemma}
\begin{proof}
By Proposition \ref{fix}, there exists $\psi\in \cC$ such that $\int_J\mu(x) \psi(x)|dx|^2\neq 0$.
Choose such $\psi$ and assume the contrary. Then the function $\mu r_\psi$ is integrable on $J$ and we can write:
$$\int_{J}\mu(x) r_\psi(x) |dx|^2=\int_{J}(f^*\mu)(x) r_\psi(x) |dx|^2=\int_{J}\mu(x) (T r_\psi)(x) |dx|^2=$$
$$\int_{J}\mu(x) [r_\psi(x) - \psi(x)] |dx|^2=
\int_{J}\mu(x) r_g(x)|dx|^2-\int_J \mu(x) g(x)|dx|^2 .$$
Hence, $\int_J \mu(x) \psi(x)|dx|^2=0$, a contradiction.
\end{proof}

\section{Rational maps with non-empty Fatou set. Proofs of the main results.}\label{PR}
\subsection{Proof of Proposition \ref{propmain}}\label{prop:proof}

(1)$\Rightarrow$ (2). 
First, it is easy to check that, for every $n$,
$$\int_{K_n'}|r_\psi(x)| |dx|^2 \le \int_J|\psi(x)||dx|^2<\infty.$$ 

Now assume that $\mu$ is an invariant line field. By Lemma \ref{psinonint}, there is $\psi\in\cC$ such that
$\int \mu(x)\psi(x)|dx|^2\neq 0$. Since $|\mu|$ takes either $1$ or $0$ a.e., we can write using Proposition \ref{psi}(2) and the duality of $T$ and $f^*$:
$$\int_{K_n'}|r_\psi(x)| |dx|^2\ge |\int_{K_n'}r_\psi(x)\mu(x) |dx|^2|=
|\int_{K_n'}\sum_{i=0}^\infty (T^i\psi)(x)\mu(x) |dx|^2|=$$
$$=|\sum_{i=0}^\infty\int_{f^{-i}(K_n')}\psi(y)\mu(y) |dy|^2|=|\int_{\cup_{i=0}^\infty f^{-i}(K_n')}\psi(y)\mu(y) |dy|^2|=$$
$$=|\int_J \psi(y)\mu(y)|dy|^2-\int_{K_n} \psi(y)\mu(y)|dy|^2.$$
Now, if $\mu$ is supported in $J\setminus K_{\pm\infty}$ (in particular, vanishes on $K_\infty$), then, as $n\to\infty$:
$$\int_{K_n} \psi(y)\mu(y)|dy|^2\to \int_{\cap_n K_n} \psi(y)\mu(y)|dy|^2=\int_{K_\infty} \psi(y)\mu(y)|dy|^2=0,$$
which implies that
$$
\liminf_{n>0}\int_{K_n'\setminus K'_\infty}|r_\psi(x)| |dx|^2\ge $$
$$\ge \liminf_{n>0}|\int_{K_n'\setminus K'_\infty}r_\psi(x)\mu(x) |dx|^2|=|\int_J \psi(y)\mu(y)|dy|^2|>0.$$

(2)$\Rightarrow$ (1). Let us fix any $\psi\in\cC$ so that (\ref{intk'n}) holds.


We'll often omit in this proof $\psi$ in the superscript writing $\sigma$, $\sigma_n$ etc instead of $\sigma^\psi$, $\sigma^\psi_n$ etc.
The following identity is crucial:
$$\int_{K'_n}|r_\psi(x)||dx|^2=\int_{K'_n}\sigma(x) r_\psi(x)|dx|^2=\int_{K'_n}\sigma(x) \sum_{i=0}^\infty (T^i\psi)(x)|dx|^2=$$
$$\sum_{i\ge 0}\int_{f^{-i}(K'_n)}\sigma_n(y)\psi(y)|dy|^2=\int_\C\sigma_n(y)\psi(y)|dy|^2.$$
So by (\ref{intk'n}),
\begin{equation}\label{nonzero}
\int_\C\sigma_n(y)\psi(y)|dy|^2\ge \delta,
\end{equation}
for some $\delta>0$ and all large $n$.


Let $\sigma_n^*=f^*\sigma_n$. Observe that $\sigma_n^*=\sigma_n$ on $\cup_{i=1}^\infty f^{-i}(K'_n)$ and $\sigma_n^*=0$ otherwise
(so $\sigma_n$ and $\sigma^*_n$ differ only on $K'_n\setminus K'_\infty$ where $\sigma_n=\sigma$ while $\sigma_n^*=0$).
Then, for any $g\in L_1(J)$,
$\int_J(\sigma_n(x)-\sigma^*_n(x))g(x)|dx|^2=\int_{K'_n\setminus K'_\infty}\sigma(x)g(x)|dx|^2\to 0$ as $n\to\infty$ because $area(K'_n\setminus K'_\infty)\to 0$. In particular, the weak* limit points of sequences $(\sigma_n)$, $(\sigma^*_n)$ are the same.

Now, let $\sigma_{n_k}\to\nu$ in the weak* topology. Then $\sigma^*_{n_k}=f^*\sigma_{n_k}\to\nu$.
Since $Tg\in L_1(J)$ for any $g\in L_1(J)$, hence, $\int_J f^*\sigma_{n_k} g = \int_J \sigma_{n_k} Tg \to \int_J \nu Tg
= \int_J f^*\nu g$, i.e. $f^*\sigma_{n_k}\to f^*\nu$ in the weak* topology. Thus $f^*\nu=\nu$, i.e. $\nu\in Fix(f^*)$.
On the other hand, $\nu\neq 0$ by (\ref{nonzero}), as $\psi\in L_1(J)$. Clearly, $||\nu||_\infty\le 1$.



\subsection{Proof of Theorem \ref{thm:main}}\label{thm:proof}

Let $\sigma^{\psi}_{E,n_k}\to\nu$ in the weak* topology, along a subsequence $n_k\to\infty$.
It follows that since $\supp(\sigma^{\psi}_{E,n})\subset E$ for all $n$, then $\supp(\nu)\subset E$.
On the other hand, as in the proof of Proposition \ref{propmain}, $\nu\in Fix(f^*)$. Therefore, by part (2) of Proposition \ref{fix} and the conditions of the theorem, $\nu=0$ on $E\setminus E_\mu$ and $\nu=t\mu$ on $E_\mu$,
for some $|t|\le 1$.
Hence,
\begin{equation}\label{extrnk}
\int_{E_\mu}\sigma^\psi_{n_k}(y)\psi(y)|dy|^2\to \int_{E_\mu} t\mu(y) \psi(y)|dy|^2.
\end{equation}
We have to prove that $t=1$.
Repeating a calculation of Proposition \ref{propmain}, we have:
$$\int_{K'_n\cap E_\mu}|r_\psi(x)||dx|^2=\int_{K'_n\cap E_\mu}\sigma^\psi(x) r_\psi(x)|dx|^2=\int_{K'_n\cap E_\mu}\sigma^\psi(x) \sum_{i=0}^\infty (T^i\psi)(x)|dx|^2=$$
$$\sum_{i\ge 0}\int_{f^{-i}(K'_n\cap E_\mu)}\sigma^\psi_n(y)\psi(y)|dy|^2=\int_{E_\mu}\sigma^\psi_n(y)\psi(y)|dy|^2.$$
Therefore,
\begin{multline}\label{basicineq}
\int_{E_\mu}\sigma^\psi_{n_k}(y)\psi(y)|dy|^2= \\ \int_{K'_{n_k}\cap E_\mu}|r_\psi(x)||dx|^2\ge |\int_{K'_{n_k}\cap E_\mu}\mu(x) r_\psi(x)|dx|^2|= \\ |\int_{K'_{n_k}\cap E_\mu}\mu(x) \sum_{i=0}^\infty T^i\psi(x)|dx|^2|=|\int_{\cup_{i=0}^\infty f^{-i}(K'_{n_k}\cap E_\mu)}\mu(y) \psi(y)|dy|^2| \\
\to |\int_{E_{\mu}}\mu(y)\psi(y)|dy|^2|
\end{multline}
where we use that $f^*\mu=\mu$ and $area(\cup_{i=0}^\infty f^{-i}(K'_{n_k}\cap E_\mu))\to area(E_\mu)$.
Coupled with (\ref{extrnk}), we obtain:
$$t\int_{E_\mu} \mu(y) \psi(y)|dy|^2\ge |\int_{E_\mu}\psi(y)\mu(y)|dy|^2|.$$
By the normalization of $\mu$, $\int_J \mu\psi=\int_{E_\mu}\mu\psi>0$, which implies that $t$ is real and, moreover, $t\ge 1$. As $|t|\le 1$,
we conclude that $t=1$, in other words,
$\sigma^\psi_{E,n}\to\nu_0$ along all subsequences, in the weak* topology.
This proves item (1).
Having this, it is easy to see that the inequality in (\ref{basicineq}) turns into an equality at the limit, which proves item (2).

(3). It is enough to prove that $\sigma^{\psi}_{E,n}\to \mu$ in measure on every bounded subset $F$ of $E_\mu\cap\C$.
By the proven part, for every $g\in L_1(E_\mu)$,
$$\int_{E_\mu}\sigma^\psi_{E,n}(x)g(x)|dx|^2\to \int_{E_\mu} \mu(x)g(x)|dx|^2.$$
Putting here $g=(1/\mu) {\bf 1}_F$,
one gets
$$\int_{F}\{1-\frac{\sigma^\psi_{E,n}(x)}{\mu(x)}\}|dx|^2\to 0$$
where the sequence of functions
$$b_n(x):=\frac{\sigma^\psi_{E,n}(x)}{\mu(x)}$$
is such that $|b_n(x)|=1$ a.e. on $F$.
This along implies
that $b_n\to 1$ in measure on each bounded subset $F$ of $E_\mu\cap\C$, i.e., the first assertion of (3).
The second, equivalent assertions follows immediately from Remark \ref{remarksigma} (2).

\subsection{Proof of Corollary \ref{coro:quadr}}

Recall that $f_c$ admits at most one invariant line field $\mu$ on $J_c$, so that $|\mu|=1$ on its support $E_\mu$ and $\mu=0$ elsewhere.
It follows that the equivalence of conditions (1)-(4a) follows from Proposition \ref{propmain} and Theorem \ref{thm:main}. Let us prove
(4b).

Let, for some $m$, $V$ be a neighbourhood $V$ of $P_m$ such that $\overline{V}\cap (P\setminus P_m)=\emptyset$.
For all $i\in\{1,...,p_m-1\}$, sets $F_c^i(P_m)$ are pairwise disjoint along with some their simply connected neighborhoods $V_i$, moreover,
all $V_i$ are disjoint with $V$.
Hence, symmetric sets $-V_i$, $i=1,...,p_n$, are pairwise disjoint simply connected neighborhoods of the corresponding compact sets $-f_c^i(P_m)$ and all are disjoint with $V$. In Particular, $W':=\cup_{i=1}^{p_m-1}$ contains $(-P_c)\setminus P_m$.
Now, fix a big enough $N$ such that $K'_N\setminus V\subset W'$. Since $W'$ consists of a finitely many simply connected components, by Koebe distortion theorem, there is $C>0$ such that, for any branch $g:=f_c^{-i}$ defined on any component $U$ of $W'$, and any compact $E\subset K'_N\cap U$ , $area(g(E))/area(g(K'_N\cap U))\le C area(E)/area(K'_N\cap U)$. On the other hand, $area(\cup_{i\ge 0}f^{-i}(K'_N)\le area(J_c)<\infty$
while $area(K'_n\setminus K'_\infty)\to 0$. This implies that $area(\cup_{i\ge 0}f^{-i}(K'_n\cap W')\to 0$. This, in turn, is enough to get the conclusion of item (4b).


\section{Rational maps with the Julia set equals the Riemann sphere}\label{rat-c}

Here we consider the alternative case: $J=\hat\C$.

This is equivalent to say that $f$ has neither rotation domains, nor attracting or parabolic cycles.
In particular, $f$ fas at least $3$ distinct fixed points. Note that they are not critical points as otherwise the Fatou set would be non-empty.
After a Mobius conjugation, we can assume that
$$f(\infty)=\infty, f(1)=1, f(0)=0.$$

We also assume that
 $f$ is not exceptional, i.e., $f$ is not double covered with by an integral torus endomorphism,

As before,
$P\subset \hat\C$ denotes the postcritical set of $f$ and
$\cV:=\{v_1,...,v_p\}\subset \C$ is the set of all geometrically different critical values of $f$ that lie in $\C$.

Let
$$\Omega=\C\setminus \{P\cup\{0,1\}\}.$$


We replace the Cauchy kernel $1/(x-z)$ which is no longer integrable in the considered case on $J=\hat\C$ by the following kernel which is common in the theory of quasiconformal maps \cite{a} (see also \cite{mak}):
$$\psi_z(x)=\frac{z(z-1)}{x(x-1)(x-z)}$$
and for which $\psi_z\in L_1(\C)$ for every $z\in \C$.

\

The list of test functions becomes
$$\cC=\{\psi_{v_j}(x), j=1,...,p\}.$$

\

We will usually assume that the following two conditions ($\ast$)-($\ast$$\ast$) hold:

($\ast$) $P\neq \hat\C$.
Note that as $J=\hat\C$, this is equivalent to saying that $P$ is nowhere dense in $\C$.

($\ast$$\ast$) $d(f^n(z), P)\to 0$ for almost every point $z\in J$
(where $d$ is the spherical distance).
This condition is not really a restriction for us, because of the following fact (Theorem 3.17 of \cite{mcm}):

{\it Let $f$ carry an invariant line field on its Julia set. Then either $f$ is exceptional, or $d(f^n(z), P)\to 0$ for a.e. $z\in J$.}

In other words, ($\ast$$\ast$) holds automatically assuming that $f$ carries an invariant line field on $J$ (and not exceptional).

Lemma \ref{kk'} as well as its proof hold under the assumptions ($\ast$)-($\ast$$\ast$)
\begin{lemma}\label{kk'-c}
Suppose that $P\neq \hat\C$ and $d(f^n(z), P)\to 0$ for almost every point $z\in J$
where $d$ is the spherical distance.
Given a small enough neighborhood $W$ of $P$ such that $area(\partial W)=0$,
let $K(W):=\{x\in W\cap J: f^n(x)\in W, n=1,2,...\}$.
Then $K(W)$ is a fundamental set.
\end{lemma}
This allows us to have a fundamental set $K(W)$ contained in any given small enough neighborhoods $W$ of $P$
with $area(\partial W)=0$.
\begin{theorem}\label{main-c}
Assume that $P\neq \hat\C$ and $d(f^n(z), P)\to 0$ for almost every point $z\in J$
(e.g. the latter holds if $f$ carries an invariant line field on its Julia set and not exceptional). By Lemma \ref{kk'-c}, for every decreasing sequence of small neighborhoods of $P$ there corresponds a sequence of fundamental sets.
Let us fix such a decreasing sequence of fundamental sets $K_n$, $n=1,2,...$. As earlier,
$K_\infty=\cap_n K_n$, $K'_{\infty}=f^{-1}(K_\infty)\setminus K_\infty$
and $K_{\pm\infty}=\cup_{i=0}^\infty f^{-i}(K_\infty)$.

Then the main results, Proposition \ref{propmain} and Theorem \ref{thm:main}, hold literally also in the considered case $J=\hat\C$, under the assumptions and notations as above.
\end{theorem}

Proof repeats the ones of Proposition \ref{propmain} and Theorem \ref{thm:main}, with the only changes in Lemma \ref{T}, Proposition \ref{psi} and Proposition \ref{fix} of Section \ref{AS} that we describe in detail now.

Lemma \ref{T} turns into the following (cf. \cite{mak} for the case of simple critical points):
\begin{lemma}\label{T-c}
Assume $f$ is a rational function of degree at least $2$ which is normalized so that
$\infty$, $1$, $0$ are fixed points of $f$ and not critical points of $f$. Then, for
any $z\in\C$ such that $f'(z)\neq 0$ and $x\in\C$ such that $x\notin \{P\cup\{1\}\cup\{0\}\}$ and $x\neq f(z)$:
\begin{equation}\label{eqT-c}
(T\psi_z)(x)=\frac{1}{f'(z)}\psi_{f(z)}(x)+\sum_{j=1}^p L_j(z)\psi_{v_j}(x).
\end{equation}
\end{lemma}
\begin{proof}
Let
$$\Phi_{z,x}(x)=\frac{1}{f'(w)}\psi_{f(w)}(x)\psi_z(w)=
\frac{1}{f'(w)}\frac{f(w)(f(w)-1)}{x(x-1)(x-f(w))}\frac{z(z-1)}{w(w-1)(w-z)}.$$
Notice that:
\begin{enumerate}
\item $\Phi(w)=O(1/w^2)$ as $w\to\infty$,
\item poles of $f$ in $\C$ are zeroes of $\Phi$ and $w=0,1$ are regular points of $\Phi$.
\end{enumerate}
For fixed $z, x\in\C$ and $R>0$ big enough, we calculate
$$I=\frac{1}{2\pi i}\int_{|w|=R}\Phi_{z,x}(w)dw.$$
On the one hand, by item (1), $I=0$ for all $R$ big enough.
On the other, using item (2):
$$I(x,z)=-(T\psi_z)(x)+\frac{1}{f'(z)}\psi_{f(z)}(x)+\sum_{c: f'(c)=0}I_c(z,x)$$
where
$$I_c(z,x)=\frac{1}{2\pi i}\int_{|w-c|=\epsilon}\frac{1}{f'(w)}\psi_{f(w)}(x)\psi_z(w)dw,$$
for some small enough $\epsilon>0$.
For any such $c$, $f(w)=v+D(w-c)^{m+1}+O(w-c)^{m+2}$ where $v=f(c)$, $D\neq 0$ and $m=m_c>0$ is the multiplicity of $c$.
Let
$$g(w)=\frac{(w-c)^m}{f'(w)}$$
so that $g$ is holomorphic near $c$ and $g(c)\neq 0$.
Repeating a calculation as in \cite{Lpert}, we get:
$$I_c(x,z)=z(z-1)B_{m-1}(c,z)\psi_v(x)$$
where $B_l=B_l(c,z)$ are defined by the expansion:
$$\sum_{l=0}^\infty B_l(w-c)^l=\frac{g(w)}{w(w-1)(w-z)}.$$
Then
$$L_j(z)=\sum_{c\in\C: f'(c)=0, v_j=f(c)}z(z-1)B_{m_c-1}(c, z)$$
which gives us (\ref{eqT-c}).
\end{proof}
Proposition \ref{psi} is now replaced by
\begin{prop}\label{psi-c}
Suppose that ($\ast$)-($\ast$$\ast$) hold (in particular, the open set $\Omega$ is dense in $\C$).
Let
$\psi\in\cC$, i.e., $\psi=\psi_v$, for some $v\in \cV$. Then we have:

(1) the series
$|r|_{\psi}=\sum_{i=0}^\infty (|T|^i \psi)(x)$
converge locally uniformly to continuous functions in
$\Omega$;

(2) the function $r_\psi$ is well-defined and holomorphic in $\Omega$, and
$$r_\psi(x)=\sum_{i=0}^\infty (T^i \psi)(x)$$
where the series converges locally uniformly in $\Omega$.
\end{prop}
\begin{proof}
It is enough to prove that, for any disk $B=B(x_0, 2\rho)$ such that $B\subset \Omega$, the series $|r|_{\psi}$ converges and uniformly bounded in $B_0:=B(x_0, \rho)$. Choose a fundamental set $K=K(W)$ where $W$ in the $\rho$-neighborhood of $P$.
Then we follow essentially the proof of Lemma \ref{psi}. Namely, let $K'=f^{-1}(K)\setminus K$.
Choose $j\ge 0$ such that $area(B_0\cap f^{-j}(K'))>0$. As
$$\int_{f^{-j}(K')}\sum_{i=0}^\infty (|T|^i |\psi|)(x)|dx|^2=\int_{\cup_i f^{-i}(f^{-j}(K'))}|\psi(x)||dx|^2:\le \int_{\C}|\psi(x)||dx|^2<\infty,$$
it follows that the series $|r|_\psi$ converges a.e. on $f^{-j}(K')$, in particular, for some $x_1\in B_0$.
Then we repeat the second part of the proof of Lemma \ref{psi}, with the only difference that now
$|\psi(x)|/|\psi(x_1)|\le C_2^3$, for every $x\in B_0$.
\end{proof}
Finally, Proposition \ref{fix} and its proof hold almost literally in the considered case as well; minor changes are needed only in part (1) where we replace the definition of the map $\cF$ (and then apply Lemma \ref{T-c} instead of Lemma \ref{T-c}) as follows.  

For $\nu\in L_\infty(\C)$, let
$$\hat\nu^\infty(z)=\int\nu(x)\psi_z(x)|dx|^2.$$

\begin{prop}\label{fix-c}

(1)
The following map $\cF^\infty: Fix(f^*)\to\C^{p}$:
$$\nu\in Fix(f^*)\mapsto F(\nu)=(\hat\nu^\infty(v_1),...,\hat\nu^\infty(v_p)$$
is linear and injective.
(2) Part (2) of Proposition \ref{fix} holds.
\end{prop}

As a conclusion, for every $\nu\in Fix(F^*)\setminus \{0\}$, there exists $\psi\in\cC$ such that
$$\int \psi\nu\neq 0.$$

\section{An application}\label{app}

Let $f(z)=z^2+c$ be infinitely renormalizable and $p_n\to\infty$ be an infinite sequence of periods of simple renormalizations of $f$ \cite{mcm}.
Given $n\ge 1$, let $J_n$ be the cycle of "small" Julia sets corresponding to the period $p_n$, that is,
$J_n=\cup_{i=0}^{p_n-1} f^i(J_n(0))$ where $J_n(0)\ni 0$ is the small Julia set containing zero.


Recall that $J_n(0)$ is the (connected) filled Julia set
of some polynomial-like map $f_c^{p_n}:U_n\to V_n$ around a single critical point at $0$ where $U_n, V_n$ are Jordan disks with $\overline{U_n}\subset V_n$).

We choose a decreasing sequence of fundamental sets to be 
$$J_1, J_2,...,J_n,....$$ 
Then
$J_\infty:=\cap_{n}J_n$,
$$J_n':=f^{-1}(J_n)\setminus J_n=(-J_n\setminus) J_n(0)=\cup_{i=1}^{p_n-1} J'_n(i)$$
where $J'_n(i)=-J_n(i)$, and
$J'_\infty:=f^{-1}(J_\infty)\setminus J_\infty=(-J_\infty)\setminus J_\infty(0)$
and $J_\infty(0)$ is the component of $J_\infty$ conraning $0$.

Let $\mu$ be an invariant line field of $f$ such that its support $E=\supp(\mu)$ is disjoint with $\cup_{i\ge 0}f^{-i}(J_\infty)$.
As usual, we normalize $\mu$ such that
$$\hat\mu(c):=\int\frac{\mu(x)}{x-c}|dx|^2>0.$$

Recall that $\sigma=|r|/r$.

Introduce several notations:
\begin{itemize}
\item $E_n=J'_n\cap E$ and $E_{n,i}=J'_{n,i}\cap E$.
\item Given $\epsilon>0$, let
$$X_{n}(\epsilon)=\{x\in E_{n}: |1-\frac{\mu(x)}{\sigma(x)}|<\epsilon\}, \\ X_{n,i}(\eps)=X_n(\eps)\cap J'_{n,i}, \\ i=1,...,p_n-1.$$
\item For $\eps>0$ and $\delta\in (0,1)$,
$$I_n^\eps(\delta)=\{i\in \{1,...,p_n-1\}: \int_{X_{n,i}(\eps)}|r(x)|>(1-\delta)\int_{E_{n,i}}|r(x)|\}.$$
\end{itemize}

By $x_n\sim y_n$ we mean that $\lim_{n\to\infty}x_n=\lim_{n\to\infty}y_n$. Also, $(,,,)^c$ means the complement of (...) (details will be clear from the context).

The following claim is general (i.e. holds with no any assumptions on renormalizations):
\begin{lemma}\label{l1}
\begin{enumerate}\label{l:enum}
\item For every $\eps>0$,
$$\lim_{n\to\infty}\int_{X_n(\epsilon)}|r(x)| |dx|^2=\lim_{n\to\infty}\int_{E_n} |r(x)| |dx|^2=\hat\mu(c).$$
\item For every $\eps>0$ and $0<\delta<1$,
$$\lim_{n\to\infty}\sum_{i\in I_n^\eps(\delta)}\int_{X_{n,i}(\eps)}|r|=\hat\mu(c).$$
\end{enumerate}
\end{lemma}
\begin{proof} As $|\sigma|=1$ and $|\mu|=1$ on $E$, in view of the identity
$|1-b|=(2\Re(1-b))^{1/2}$ if $|b|=1$, it is enough to show the claim of item (1) replacing modulus in the definition of $X_n(\eps)$ by the real part (and replacing $\eps$ by $\eps^2/2$ but we denote the latter again by $\eps$).
We fix $\eps, \delta$ and then omit them in the notations $X_n(\eps)$ and $I_n^\eps(\delta)$.
By the chain of limits of Corollary \ref{coro:quadr}(2),
$$0\sim\int_{E_n}|r|\Re(1-\frac{\mu}{\sigma})\ge \int_{(X_n)^c}|r|\Re(1-\frac{\mu}{\sigma})\ge
\eps\int_{(X_n)^c}|r|.$$
Along with
$\int_{X_n}|r| + \int_{(X_n)^c}|r|=\int_{E_n}|r|\sim\hat\mu(c)$, this implies item (1).
Using (1), we proceed with fixed $\eps,\delta$:
$$\int_{E_n}|r|\sim\int_{X_n}|r| \le (1-\delta)\sum_{i\in (I_n)^c}\int_{E_{n,i}}|r|
+\sum_{i\in I_n}\int_{E_{n,i}}|r|=\int_{E_n}|r|-\delta\sum_{i\in (I_n)^c}\int_{E_{n,i}}|r|.$$
As $\delta>0$ is fixed,
$$\sum_{i\in (I_n)^c}\int_{X_{n,i}}|r| \le \sum_{i\in (I_n)^c}\int_{E_{n,i}}|r|\to 0$$
as $n\to\infty$. The item (2) follows.

\end{proof}

From now on, we {\it assume that $f$ has unbranched complex bounds}: one can choose a sequence of simple renormalizations with periods $p_n\to\infty$
and polynomial-like (PL) maps $f^{p_n}: U_n\to V_n$ having its Julia set $J_n(0)$ in such a way that
$\liminf_n \mod(V_n\setminus U_n)>0$ and $V_n\cap P=J_n(0)\cap P$.

\begin{theorem}\label{mcm} (\cite{mcm}, Theorem 10.2)
$f$ with unbranched complex bounds carries no an invariant line field on its Julia set.
\end{theorem}

Assume, by a contradiction, that $\mu$ is any invariant line field of $f$.

Our aim is to prove the theorem without using Lebesgue's almost continuity theorem for the function $\mu\in L_\infty(\C)$.

By \cite{mcm}, Theorem 10.7,
$$J_\infty=\cap_{n\ge 0}J_n=P$$
and $P$ has measure zero. In particular, the support of $\mu$ must be disjoint with $\cup_{i\ge 0}f^{-i}(J_\infty)$
simply because the latter set has measure zero. Thus
Corollary \ref{coro:quadr} and Lemma \ref{l1} apply.

By \cite{mcm}, \cite{mcm1}, one can "improve" domains $U_n, V_n$ as follows (e.g., by pulling back each PL map $f^{p_n}: U_n\to V_n$
several times if necessary and after that renaming $W_n:=V_n$ and replacing $V_n$ by a bounded domain with the boundary to be the core curve of the fundamental annulus of $V_n\setminus \overline{U_n}$): for some $m>0$, $L>!$ and every $n$ along a subsequence,
$m\le \mod(V_n\setminus \overline{U_n})\le 2m$, there is a simply connected $W_n\supset \overline{V_n}$ such that 
$\mod(W_n\setminus \overline{V_n})\ge m$, and $\partial V_n$ is an $L$-quasi-circle.

Now, for each $n$, pull back the pair $V_n, U_n$ along the orbit $0,f(0),...,f^{p_n-1}(0)$ getting
domains $V_n(i), U_n(i)$, $i=0,...,p_n$ (where $V_n(0)=V_n$, $U_n(0)=U_n$) such that $f^{p_n}: V_n(i)\to U_n(i)$ is a polynomial-like map with the Julia set
$J_n(i)=f^{p_n-i}(J_n(0))$ and $V_n(i)\cap P=J_n(i)\cap P$.
Then $V_n(i)':=-V_n(i)$ is disjoint with $P$, for $i=1,...,q_n-1$.
Moreover, all branches of $f^{-k}$ for all $k\ge 0$ have uniformly bounded distortions
being restricted to $V'_n(i)$, for all considered $n$ and all $0<i<p_n$.

As $V'_n(i)$ are disjoint with $P$, the function $r$ is holomorphic
in all $V'_n(i)$, $i\neq 0$.
Note that $X_{n,i}\subset E_{n,i}\subset V'_{n}(i)$ and introduce
$$\gamma_{n,i}=\frac{\int_{V_n(i)}|r|}{\int_{E_{n,i}}|r|}, i=1,...,p_n-1.$$
Given $K>1$, let
$$M_n(K)=\{0<i<p_n: \gamma_{n,i}\le K\}.$$

\begin{lemma}\label{l2}
\begin{enumerate}
\item There exists $L>0$ such that, for all $n$,
$$\sum_{i=1}^{p_n-1}\int_{V'_{n}(i)}|r|\le L\sum_{i=1}^{p_n-1}\int_{E_{n,i}}|r|.$$
\item For $N>\max\{1, \hat\mu(c)/L\}$, let $K=NL/\hat\mu(c)$. Then
$$\sum_{i\in M_n(K)}\int_{E_{n,i}}|r|>\frac{(N-1)L}{\frac{NL}{\hat\mu(c)}-1}.$$
\end{enumerate}
\end{lemma}
\begin{proof}
Part (1) follows as in the proof of part (1) of Proposition \ref{psi} since all branches of $f^{-k}$ for all $k\ge 0$ have uniformly bounded distortions on all $V'_n(i)$,

For part (2),
$$\sum_{i=1}^{p_n-1}\int_{E_{n,i}}|r|\sim\hat\mu(c),$$
then
$$L\ge \sum_{i=1}^{p_n-1}\gamma_{n,i}\int_{E_{n,i}}|r|\ge \frac{NL}{\hat\mu(c)}\sum_{i\in (M_n(K))^c}\int_{E_{n,i}}|r|
+\sum_{i\in M_n(K)}\int_{E_{n,i}}|r|\sim$$
$$\frac{NL}{\hat\mu(c)}\hat\mu(c)- (\frac{NL}{\hat\mu(c)}-1)\sum_{i\in M_n(K)}\int_{E_{n,i}}|r|.
$$
The bound of part (2) follows.
\end{proof}

From now on, fix $K_*$ such that, for the corresponding $N_*=K_*\hat\mu(c)/L$,
$$\frac{(N_*-1)L}{\frac{N_*L}{\hat\mu(c)}-1}>\frac{1}{2}\hat\mu(c).$$
Then
$$\sum_{i\in M_n(K)}\int_{E_{n,i}}|r|>\frac{1}{2}\hat\mu(c).$$

Below, we pass sometimes to subsequences of $n$'s without warning.

\begin{lemma}\label{l:key}
There exist positive sequences $\eps_n\to 0, \delta_n\to 0$ and, for each $n$, an index $i_n\in\{1,...,p_n-1\}$ such that the following hold:
$$\int_{V'_{n}(i_n)}|r|<K_*\int_{E_{n,i_n}}|r| \mbox{ and } \int_{X_{n,i_n}(\eps_n)}|r|>(1-\delta_n)\int_{E_{n,i_n}}|r|.$$

\end{lemma}
\begin{proof}
The first integral inequality means that
$i_n\in M_n(K_*)$ and
the second one - that $i_n\in I_n^{\eps_n}(\delta_n)$.
Hence, it is enough to prove that for each small positive fixed $\eps,\delta$ there is an arbitrary big $n$ such that
$M_n(K_*)\cap I_n^{\eps}(\delta)\neq\emptyset$ (then the required sequences $\eps_n,\delta_n$ are obtained by induction).
Fix $\eps,\delta$ and assume the contrary, i.e., for all $n$ big enough two subsets $M_n(K_*)$, $I_n:=I_n^\eps(\delta)$ of $\{1,...,p_n-1\}$ are disjoint. Using Lemmas \ref{l1}-\ref{l2}, we have as $n\to\infty$:
$$\hat\mu(c)\sim \sum_{i=1}^{p_n-1}\int_{E_{n,i}}|r|\ge \sum_{i\in I_n}\int_{E_{n,i}}|r|+\sum_{i\in I_n}\int_{M_n(K_*)}|r|\ge \frac{3}{4}\hat\mu(c)+\frac{1}{2}\hat\mu(c),$$
a contradiction.
\end{proof}

{\bf Proof of Theorem \ref{mcm}.}

Apply Lemma \ref{l:key} to get sequences $\eps_n, \delta_n, i_n$.
The map
$$P_n:=-f^{q_n}: U'_{n}(i_n)\to V'_{n}(i_n)$$ 
is a PL map with a single critical point at some $x'_n\in J'_{n,i_n}$,

We normalize maps $P_n$ and corresponding sets by linear changes 
$A_n(z)=\alpha_n z+x'_n$ where $\alpha_n=\diam(V_n(i_n))$.


We mark usually objects lifted by $A_n$ by $\tilde{.}$. In particular,

$\tilde V_n:=A_n^{-1}(V'_n(i_n))$, $\tilde U_n:=A_n^{-1}(U'_n(i_n))$, $\tilde J_n:=A_n^{-1}(J_{n}(i_n))$, $\tilde E_n:=A_n^{-1}(E_{n,i_n})$
and
$\tilde X_n:=A_n^{-1}(X_{n,i_n})$.

Then we lift the function $r$ restricted to domains $V'_{n,i_n}$ and normalize the lifts as follows:

$\tilde r_n(z)=c_n r(A_n(z))$ where
$$c_n=\frac{\alpha_n^2}{\int_{V'_n(i_n)}|r(x)||dx|^2},$$
is chosen so that
$$\int_{\tilde V_n}|\tilde r|=1$$
Denote, for the future use,
$$C_n:=\frac{c_n}{\alpha_n^2}.$$

Let 
$$g_n:=A_n^{-1}\circ P_n \circ A_n: \tilde U_n\to \tilde V_n.$$



Along a subsequence, $(\tilde V_n, 0)\to (\tilde V, 0)$, $(\tilde U_n, 0)\to (\tilde U, 0)$ in the Caratheodory topology
and $g_n\to g:\tilde U\to \tilde V$ uniformly on compacts in $\tilde U$.
In particular, for every compact $\tilde K\subset\tilde V$ there are $n_0$ and $\rho>0$ such that $\tilde K\subset \tilde V_n$ and
$\dist(\tilde K,\partial(\tilde V_n))\ge \rho$ (in the Euclidean distance).

As $\tilde r_n$ is holomorphic in $\tilde V_n$, then for every $z\in \tilde K$,
$$|\tilde r_n(z)|\le \frac{1}{\pi\rho^2}\int_{B(z,\rho)}|\tilde r_n|\le \frac{1}{\pi\rho^2}$$

This implies that (along a subsequence, as usual)
$$\tilde r_n\to\tilde r$$
uniformly on compacts in $\tilde V$ where $\tilde r$ is a holomorphic function in $\tilde V$.
Let us prove that
$$\tilde r\not\equiv 0.$$

Indeed,

$$1=\int_{\tilde V_n}|\tilde r_n|=C_n\int_{V'_n(i_n)}|r|\le K_* C_n\int_{E_{n,i_n}}|r|=K_*\int_{\tilde E_n}|\tilde r_n|$$
so that
\begin{equation}\label{tildernot0}
\int_{\tilde E_n}|\tilde r_n|\ge \frac{1}{K_*}
\end{equation}

Now, as $\partial\tilde V$ are all $L$-quasi-circles, closures of $\tilde V_n$ tend to the closure of $\tilde V$
also in the Hausdorff distance and $\partial\tilde V$ is an $L$-quasi-curcle \cite{mcm1}.
Besides, $\mod(\tilde V_n\setminus \overline{\tilde U_n})\ge m$ for all $n$ and $\tilde E_n\subset \tilde U_n$. 
Therefore, $\tilde E_n\subset \tilde X$, for all $n$ and some compact subset $\tilde X$ of $\tilde V$.
As $\tilde r_n\to \tilde r$ uniformly on compacts in $\tilde V$, (\ref{tildernot0}) implies that $\tilde r\not\equiv 0$.

The lower bound (\ref{tildernot0}) implies also that, for big $n$,
$$area(\tilde E_n)\ge S:=\frac{1}{2K_* \max\{|\tilde r(z)|: z\in \tilde X\}}$$
On the other hand,
\begin{equation}\label{bound}
\int_{\tilde E_n\setminus \tilde X_n}|\tilde r_n|=C_n \int_{E_{n,i_n}\setminus X_{n,i_n}}|r|<\delta_n C_n\int_{E_{n,i_n}}|r|=
\delta_n \int_{\tilde E_n}|\tilde r_n|<\delta_n\int_{\tilde V_n}|\tilde r_n|=\delta_n\to 0.
\end{equation}

As $\tilde r_n\to \tilde r$ uniformly on compacts while $\tilde r\not\equiv 0$ and holomorphic, then,
on each small neighborhood $W$ of a point $z$ with $\tilde r(z)\neq 0$, moduli
$|\tilde r_n|$ are uniformly away from zero.
Hence, by (\ref{bound}),
$$area((\tilde E_n\setminus \tilde X_n)\cap W)\to 0$$
for each such $W$.
We conclude that
$$area(\tilde E_n\setminus \tilde X_n)\to 0.$$

Now, $J'_{n,i_n}$ is the filled Julia set of the PL map $P_n:=-f^{q_n}: U'_{n}(i_n)\to V'_{n}(i_n)$, hence,
$P_n^{-1}(J'_{n,i_n})=J'_{n,i_n}$. Since $E_{n,i_n}=E_\mu\cap J'_{n,i_n}$, then also
$P_n^{-1}(E_{n,i_n})=E_{n,i_n}$. Hence, after lifting,
$$g_n^{-1}(\tilde E_n)=\tilde E_n.$$
Then, as $n\to\infty$,
$$area(\tilde E_n\setminus g_n^{-1}(\tilde X_n))=area(g_n^{-1}(\tilde E_n\setminus \tilde X_n))\le
2\sup_{z\in \tilde E_n}|g_n'(z)|^2 area(\tilde E_n\setminus \tilde X_n)\to 0.$$
Thus we obtain:
$$area(g_n^{-1}(\tilde X_n))\sim area(\tilde X_n)\sim area(\tilde E_n)$$
while $area(\tilde E_n)\ge S>0$ for big $n$.
Therefore, for every $n$ big enough,
\begin{equation}\label{area}
area(Y_n)\ge \frac{2}{3}S,
\end{equation}
where
$$Y_n=g_n^{-1}(\tilde X_n)\cap \tilde X_n=\{z: z\in \tilde X_n, g_n(z)\in \tilde X_n\}.$$

Let $\tilde\mu_n:=\mu\circ A_n \frac{|A_n'|^2}{(A_n')^2}$ be the lift of $\mu$ from $V'_{n,i_n}$ to $\tilde V_n$..

As
$$A_n(\tilde X_n)=X_{n,i_n}(\eps_n)=\{x\in E_{n,i_n}: |\mu(x)-\frac{|r(x)}{r(x)}|<\eps_n\},$$
then
\begin{equation}\label{X}
\tilde X_n=\{z\in \tilde E_n: |\tilde\mu_n-\frac{|\tilde r_n(z)|}{\tilde r_n(z)}|<\eps_n\}.
\end{equation}
As $(-f^{p_n})^*\mu=\mu$, this translates, after lifting by $A_n$, to:
\begin{equation}\label{tildemu}
\tilde\mu_n=(\tilde\mu_n \circ g_n)\frac{|g_n'|^2}{(g_n')^2}.
\end{equation}
Now, for each $n$ choose any $z_n\in Y_n$, i.e., $z_n\in \tilde X_n$ and $g_n(z_n)\in \tilde X_n$.
If we apply (\ref{X}) to each point $z_n, g_n(z_n)$ and use (\ref{tildemu}), we get, after a simple algebra, that
\begin{equation}\label{fin}
|\frac{|\tilde r_n(z_n)|}{\tilde r_n(z_n)}-\frac{|\tilde r_n(g_n(z_n))(g_n'(z_n))^2|}{\tilde r_n(g_n(z_n))(g_n'(z_n))^2}|<2\eps_n.
\end{equation}
Let $Y$ be be set of all limit points of sequences $z_n\in Y_n$, $n=1,2,...$. Notice at the moment that $Y$ is compactly contained in $\tilde U$.
As $\eps_n\to 0$, passing to limits in (\ref{fin}), we get that
\begin{equation}\label{1steq}
\frac{|\tilde r(z)|}{\tilde r(z)}=\frac{|\tilde r(g(z))(g'(z))^2|}{\tilde r(g(z))(g'(z))^2} \ \mbox{ if } \ z\in Y\setminus \{|\tilde r|=0\}.
\end{equation}
Let us rewrite the latter equality as follows. Introduce a holomorphic in $\tilde U \setminus \{\tilde r=0\}$ function
$$\Phi:=\frac{(\tilde r\circ g)(g')^2}{\tilde r}.$$
Then (\ref{1steq}) turns into
\begin{equation}\label{fin}
\frac{|\Phi(z)|}{\Phi(z)}=1 \ \mbox{ if } \ z\in Y\setminus \{\tilde r=0\}.
\end{equation}
On the other hand, by (\ref{area}), for every big $n$ there is a compact $Z_n\subset Y_n$ such that $area(Z_n)\ge S/2$.
Then the set of limit points in a Hausdorff limit along a subsequence of compacts $Z_n$ has a measure at leat $S/2$. 
As $Z_n\subset Y_n$, this implies that $area(Y)\ge S/2$.
This, along with (\ref{fin}) and the fact that $\Phi$ is a holomorphic non-zero function in $\tilde U\setminus \{\tilde r|=0\}$ implies the existence of some $\kappa>0$,
such that
$$\tilde r(z)=\kappa\tilde r(g(z))(g'(z))^2 \ \mbox { for all } \ z\in \tilde U.$$
Iterating this identity, we get $\tilde r=0$ on
$\cup_{i>0}\{z: (g^i)'(z)=0\}$. As the latter set is infinite and compactly contained in $\tilde U$,
then $\tilde r\equiv 0$, a contradiction.

\end{document}